\documentclass[10pt,reqno]{amsart}
\usepackage{amsmath,amssymb,mathrsfs,stackengine}
\usepackage{xcolor}
\colorlet{mdtRed}{red!50!black}
\definecolor{dblue}{rgb}{0,0,.6}
\usepackage[colorlinks]{hyperref}
\hypersetup{colorlinks,linkcolor={red},citecolor={blue},urlcolor={red}}
\usepackage[all]{xy}
\usepackage{tikz,tikz-cd,tkz-graph,enumerate}

\DeclareMathOperator{\Pic}{\textnormal{Pic}}
\DeclareMathOperator{\Br}{\textnormal{Br}}

\DeclareMathOperator{\Malpha}{\mc{M}^{\boldsymbol{m,\alpha}}_X(r,\ \xi)}
\DeclareMathOperator{\Nalpha}{\mc{N}^{\boldsymbol{m,\alpha}}_X(r,\delta)}

\DeclareMathOperator{\MLalpha}{\left(\mc{M}^{\boldsymbol{m,\alpha}}_X(r,\ \xi)\right)^L}

\DeclareMathOperator{\mgpar}{\mc{M}_G^{par}}

\newcommand{\mf}[1]{\mathfrak{#1}}
\newcommand{\mc}[1]{\mathcal{#1}}
\newcommand{\bb}[1]{\mathbb{#1}}

\renewcommand{\ker}{\mathrm{Ker}}

\newtheorem{theorem}{Theorem}[section]
\newtheorem{lemma}[theorem]{Lemma} 
\newtheorem{proposition}[theorem]{Proposition}
\newtheorem{corollary}[theorem]{Corollary}

\newtheoremstyle{myroman}
{3pt}   
{3pt}   
{}      
{}      
{\normalfont} 
{.}     
{0.5em} 
{}      

\theoremstyle{myroman}
\newtheorem{definition}[theorem]{Definition}
\newtheorem{remark}[theorem]{Remark}

\numberwithin{equation}{section} 

\begin{document}
	
	\baselineskip=15.5pt 
	
	\title[Brauer group of moduli stacks of parabolic principal bundles]{Brauer group of moduli stacks of parabolic principal bundles over a curve}
	
	\author[I. Biswas]{Indranil Biswas}
	
	\address{Mathematics Department, Shiv Nadar University, NH91, Tehsil Dadri,	Greater Noida, Uttar Pradesh 201314, India}
	
	\email{indranil.biswas@snu.edu.in, indranil29@gmail.com}
	
	\author[S. Chakraborty]{Sujoy Chakraborty}
	
	\address{Department of Mathematics, 
		Indian Institute of Science Education and Research Tirupati, Andhra Pradesh 517507, India}
	\email{sujoy.cmi@gmail.com}

	
	
	\begin{abstract}
		We prove that the Brauer group of the moduli stack of parabolic stable
		principal $\textnormal{PGL}(r,\bb{C})$-bundles on a 
		curve $X$, for a generic system of weights along an arbitrary parabolic divisor, coincides with the Brauer group 
		of the smooth locus of the corresponding coarse moduli space of parabolic stable principal
		$\textnormal{PGL}(r,\bb{C})$-bundles. We also show that for any simple and simply connected complex linear 
		algebraic group $G$, the analytic and algebraic Brauer groups of the moduli stack of quasi-parabolic principal 
		$G$-bundles on $X$ vanish.
	\end{abstract}
	\maketitle
	
	\section{Introduction}
	
	The cohomological Brauer group of a quasi-projective variety $Y$ over $\bb{C}$, denoted by $\Br(Y)$, is 
	defined to be the torsion part $H^2_{\text{\'et}}(Y,\,\bb{G}_m)_{\rm torsion}$. When $Y$ is smooth, it is known that 
	$H^2_{\text{\'et}}(Y,\,\bb{G}_m)$ is already torsion. Brauer groups are interesting invariants to study for a 
	number of reasons. It is a stable birational invariant for smooth projective varieties, 
	making it very useful in studying rationality questions. In fact, Brauer group has been used in constructing 
	examples of non-rational varieties by many, including Colliot-Th\'el$\grave{e}$ne, Saltman, Peyre and 
	others. It also plays a central role in Brauer-Manin obstruction theory, which deals with the study of 
	rational points on varieties defined over a number field.
	
	The study of Brauer groups in the context of moduli of parabolic vector bundles over curves has been carried 
	out in recent times by several authors \cite{BB23,BCD23,BCD24,BD11}. The computation of the Brauer group of 
	the moduli space of stable parabolic bundles for structure groups $\text{SL}(r,\bb{C})$, 
	$\textnormal{PGL}(r,\bb{C})$ and $\textnormal{Sp}(r,\bb{C})$ have been carried out earlier in \cite{BD11, 
		BCD23, BCD24}. Our aim here is to generalize an earlier result of \cite{BCD23} for the
	case of $\text{PGL}(r,\bb{C})$-bundles, as well as study the general case for any simple and simply
	connected linear algebraic group $G$ over $\bb{C}$.
	
	More precisely, suppose $X$ is a smooth projective curve of genus $g$ defined over $\bb{C}$, with $g
	\,\geq\, 2$. Fix a line 
	bundle $\xi$ on $X$, and systems of multiplicities and weights $(\boldsymbol{m,\,\alpha})$ along a set of 
	parabolic points on $X$ (see Section \ref{section:preliminaries} for details). Let 
	$\mf{N}_X^{\boldsymbol{m,\alpha}}(r,\delta)$ denote the moduli stack of parabolic stable principal
	$\text{PGL}(r,\bb{C})$-bundles on $X$ of topological type $\delta \ \equiv\ \deg(\xi)\ \,(\text{mod}\ r)$ (see 
	Section \ref{section:brauer-group-moduli-stack}), and let $\Nalpha$ denote the corresponding coarse moduli space 
	of stable parabolic principal $\textnormal{PGL}(r,\bb{C})$-bundles. Our first main result is the generalization of 
	\cite[Theorem 1.1 (1)]{BCD23} for \textit{any} generic system of weights (earlier it was shown only for 
	\textit{full-flag} systems of weights).
	
	\begin{theorem}[{Theorem \ref{thm:brauer-group-of-moduli-stack}}]
		Let $\boldsymbol{\alpha}$ be a generic system of weights. Let $\Nalpha^{sm}$ denote the smooth locus of $\Nalpha$. Then,
		$$\Br\left(\mf{N}^{\boldsymbol{m,\alpha}}_X(r,\delta)\right)\,\ \simeq\,\
		\Br\left(\Nalpha^{sm}\right).$$
	\end{theorem}
	
	Let $G$ be a simple and simply connected linear algebraic group
	defined over $\bb{C}$. Fix a sequence $\underline{p}\,=\, (p_1,\, p_2 ,\, \cdots,\, p_n)$ of points on $X$, as well
	as a sequence $\underline{P}\,=\,(P_1,\, P_2,\, \cdots,\, P_n)$ of parabolic subgroups of $G$. Let $\mgpar\left(\underline{p},\ \underline{P}\right)$ denote the moduli stack of quasi-parabolic principal $G$-bundles on $X$ of type $\underline{P}$ (see \cite{LS97} for details). 
	
	\begin{theorem}[{Proposition \ref{prop:analytic-brauer-group-vanishing} and Theorem
			\ref{thm:brauer-group-vanishing}}]
		Both the analytic and the algebraic Brauer group of $\mgpar\left(\underline{p}, \ \underline{P}\right)$ vanish.
		In other words, if $\mc{O}^{^{an}}_{\mgpar\left(\underline{p}, \ \underline{P}\right)}$ is
		the sheaf of analytic functions on $\mgpar\left(\underline{p}, \ \underline{P}\right)$, then
		$$H^2\left(\mgpar,\ (\mc{O}^{^{an}}_{\mgpar})^*\right)_{torsion}\ =\ 0\ =\
		\Br\left(\mgpar\left(\underline{p}, \ \underline{P}\right)\right).$$
	\end{theorem}
	
	\section{Preliminaries on parabolic vector bundles}\label{section:preliminaries}
	
	Let $X$ be a connected smooth projective curve, defined over $\bb{C}$, of genus $g$, with $g\,\geq\, 2$. Fix a 
	finite subset $S\,\subset\, X$, which will be referred to as the \textit{parabolic points}. The set 
	$S$ shall remain fixed throughout.
	
	\begin{definition}\label{def:parabolic-bundles}
		A \textit{parabolic vector bundle} of rank $r$ on $X$ is an algebraic vector bundle $E$
		on $X$ of rank $r$ together with a weighted flag on the fiber $E_p$ of $E$ over every $p\,\in\, S$:
		\begin{align}\label{eqn:filtration-data}
			E_{_p} \,=\, E_{_{p,1}}\,&\supsetneq\, E_{_{p,2}}\,\supsetneq\, \cdots \,\supsetneq\, E_{_{p,\ell(p)}}
			\,\supsetneq\, E_{_{p,\ell(p)+1}} \,=\,0\\
			0\,\leq\, &\alpha_{_{p,1}}\,<\,\alpha_{_{p,2}} \,<\,\cdots\,<\,\alpha_{_{p,\ell(p)}}\,<\,\alpha_{_{p,\ell(p)+1}}=1,\nonumber
		\end{align}
		where $\alpha_{_{p,i}}\, \in\, \mathbb R$. 
		\begin{enumerate}[$\bullet$]
			\item Such a flag is said to be of length $\ell(p)$, and the numbers $m_{_{p,i}} \,:=\, \dim E_{_{p,i}} -
			\dim E_{_{p,i+1}}$ are called the \textit{multiplicities} of the flag at $p$. More precisely,
			$m_{_{p,i}}$ is the multiplicity associated to the weight $\alpha_{_{p,i}}$.
			
			\item The flag at $p$ is said to be \textit{full} if $m_{_{p,i}} \,=\,1$ for every $i$, in 
			which case clearly we have $\ell(p) \,=\, r$.
			
			\item The collection of real numbers $\boldsymbol{\alpha}\,:=\,\{(\alpha_{_{p,1}}
			\,<\,\alpha_{_{p,2}}\,<\,\cdots\,<\,\alpha_{_{p,\ell(p)}})\}_{p\in S}$ is called a system
			of weights, and the collection of integers $\boldsymbol{m}\ :=\ \{(m_{_{p,1}},\ m_{_{p,2}},\,\cdots,\,m_{_{p,\ell(p)}})\}_{p\in S}$ is called a system	of multiplicities.
			
			\item We will often denote a system of multiplicities (respectively, a system of weights) by
			$\boldsymbol{m}$ (respectively, $\boldsymbol{\alpha}$), when there is no scope of any confusion. Also,
			a parabolic vector bundle of the above type will also be denoted by $E_*$.
		\end{enumerate}
		
		If we ignore the weights and only consider the data of a vector bundle $E$ together with a given
		filtration on the fibers $\{E_p\}_{p\in S}$, then such an object will be called a
		\textit{quasi-parabolic vector bundle}. Thus, a quasi-parabolic vector bundle has an associated system
		of multiplicities, but no associated system of weights. In Section \ref{section:brauer-group-quasi-parabolic-bundles}
		quasi-parabolic bundles will appear in a more general context.
		
	\end{definition}
	\begin{definition}\label{def:parabolic-morphism}
		Let $E_*$ and $F_*$ be two parabolic vector bundles with systems of multiplicities and
		weights $(\boldsymbol{m,\,\alpha})$ and $(\boldsymbol{m',\,\alpha'})$
		respectively. A \textit{parabolic morphism} $f_* \ :\ E_* \ \longrightarrow\ F_*$
		is an ${\mathcal O}_X$--linear homomorphism $f \,:\, E\,\longrightarrow\, F$ between
		the underlying vector bundles such that at each parabolic point $p$, we have $f_{_{p}}(E^i_{_{p}})
		\,\subset\, {F}^{j+1}_{_{p}}$ whenever $\alpha^i_{_{p}}\,\geq\,{\alpha'}^j_{_{p}}$.
	\end{definition}
	
	\begin{definition}\label{def:parabolic-degree-and-slope}
		Consider a parabolic vector bundle $E_*$ of rank $r$ as in Definition \ref{def:parabolic-bundles}. The \textit{parabolic degree} of $E_*$ is defined to be the real number
		\begin{equation}
			\text{par}\deg(E_*)\ :=\ \deg(E)+\sum_{p\in S}\sum_{i=1}^{\ell(p)} m_{_{p,i}}\alpha_{_{p,i}}. 
		\end{equation}
		The \textit{parabolic slope} of $E_*$ is defined to be the real number
		\begin{equation}
			\text{par}\mu(E_*)\ :=\ \frac{\text{par}\deg(E_*)}{r}.
		\end{equation}
		
		A parabolic vector bundle $E_*$ is called \textit{parabolic semistable} (respectively,
		\textit{parabolic stable}) if for every proper sub-bundle $0\, \not=\, F\, \subset\, E$ we have 
		\[
		\textnormal{par}\mu(F_*)\,\leq\, \textnormal{par}\mu(E_*)\ \ \, (\text{respectively, }\
		\textnormal{par}\mu(F_*)\, <\, \textnormal{par}\mu(E_*)),\]
		where $F_*$ denotes the parabolic bundle defined by the parabolic
		structure on $F$ induced by the parabolic structure of $E_*$ (see \cite{MS80} for
		the details).
	\end{definition}
	
	Of course, every parabolic stable vector bundle is parabolic semistable, but the converse is not true in 
	general. Systems of weights for which parabolic semistability and parabolic stability coincide are called 
	\textit{generic} (see \cite{BY99} for more details).
	
	\subsection{Parabolic push-forward and pull-back}\label{parabolic pushforward}
	
	Let $Y$ be another irreducible smooth projective complex curve and $\gamma\,:\,  Y\,\longrightarrow\,
	X$ a finite \'etale Galois covering morphism. If $F$ is a vector bundle on $Y$ of rank $n$, then
	the direct image $\gamma_*F$ is a vector bundle on $X$ of rank $nd$, where $d$ is the 
	degree of the map $\gamma$. Given a parabolic structure on $F$ along parabolic points $S'\,\subset
	\,Y$, there is a natural way to construct a parabolic structure on $\gamma_* F$ along
	$\gamma(S')\,\subset\, X$ \cite[\S~3]{BM19}. Below, this construction will be briefly recalled.
	
	Let $S\,\subset\, X$ be a finite set of points, and suppose $F\, \longrightarrow\, Y$ is a vector bundle
	equipped with a parabolic structure
	over $\gamma^{-1}(S)\,\subset\, Y$. For convenience, consider the special case where $S\,=
	\,\{p\}\,\subset\, X$ is a singleton. Let $d\,=\,\deg(\gamma)$. Since $\gamma$ is unramified, the inverse
	image $\gamma^{-1}(p)$ consists of $d$ distinct points of $Y$. Suppose the parabolic structure of
	$F_*$ over $q\,\in\, \gamma^{-1}(p)$ is as follows:
	the parabolic data looks like
	\begin{align}\label{eqn:parabolic-structure-upstairs}
		F_q \,=\, F_{_{q,1}}\,\supsetneq\, F_{_{q,2}}\,\supsetneq \,\cdots\,\supsetneq\, F_{_{q,\ell(q)}}
		\,\supsetneq\, F_{_{q,\ell(q)+1}}\,=\,0,\\
		0\, \leq\,  \alpha_{_{q,1}}\,<\,\alpha_{_{q,2}}\,<\,\cdots<\,\alpha_{_{q,\ell(q)}}\,<\,\alpha_{_{q,\ell(q)+1}}\,=\,1.
	\end{align}
	We shall construct a parabolic structure on $E\,=\,\gamma_*F$ along the point $\{p\}$ from these data. Note that
	\[E_p\,=\,\bigoplus_{q\,\in\, \gamma^{-1}(p)}F_{q}.\]
	Define a positive integer $\ell(p)$ and real numbers $\{\beta_{_{p,1}},\,\beta_{_{p,2}},\,\cdots,\,
	\beta_{_{p,\ell(p)}}\}$ as follows: $$\{\beta_{_{p,1}},\,\beta_{_{p,2}},\,\cdots,\,\beta_{_{p,\ell(p)}}\} \ =\ \bigcup_{q\in\gamma^{-1}(p)}\{\alpha_{_{q,1}},\ \ \cdots,\ \alpha_{_{q,\ell(q)}}\}.$$
	Also assume that the $\beta_{_{p,i}}$'s form an increasing sequence
	$\left(\beta_{_{p,1}}\,<\,\beta_{_{p,2}}\,<\,\cdots\,<\,\beta_{_{p,\ell(p)}}\right)$. These will be
	the weights at $\{p\}$.
	
	For each $1\,\leq\, j\,\leq\, \ell(p)$, define a filtration of $E_p$ by the subspaces
	$$E_{_{p,j}} \,:=\, \bigoplus_{q\in\gamma^{-1}(p)}F_{_{q,t(q)}},$$
	where for each point $q\,\in\, \gamma^{-1}(S)$, the integer $t(q)\,\in\,[1,\,\ell(q)]$ (see
	\eqref{eqn:parabolic-structure-upstairs}) is the smallest one satisfying $\beta_{_{p,j}}\,\leq\, \alpha_{_{q,t(q)}}.$ Consequently,
	\begin{align*}
		E_{_{p,1}}\,\supsetneq\, E_{_{p,2}}\,\supsetneq\,\cdots\,\supsetneq\, E_{_{p,\ell(p)}}\ \supsetneq E_{_{p,\ell(p)+1}}=0\\
		0\, \leq\, \beta_{_{p,1}}\,<\,\beta_{_{p,2}}\,<\,\cdots\,<\,\beta_{_{p,\ell(p)}}\ < \ \beta_{_{p,\ell(p)+1}}=1
	\end{align*}
	is a parabolic structure on $E=\gamma_*(F)$ along the point $\{p\}$.
	
	More generally, if $S\,\subset\, X$ is a finite subset, and a parabolic vector bundle $F_*$ on $Y$ is given
	with a parabolic structure along $\gamma^{-1}(S)$, one can perform the above
	construction for each $p\,\in\, S$ to obtain a 
	parabolic push-forward $E_*\ =\ \gamma_* (F_*)$ on $X$ with parabolic structure along $S$.
	
	Let $\gamma\,:\,Y\,\longrightarrow\, X$ be as above. If $E_*$ is a parabolic vector bundle on $X$ with parabolic 
	structure along $S$, then $\gamma^*E$ has an induced parabolic structure along $\gamma^{-1}(S)$ by assigning, 
	for any $p\,\in\, S$ and $q\in \gamma^{-1}(p)$, the same filtration and weights on the fiber 
	$(\gamma^*E)_q\,=\, E_p$ as those on $E_p$.
	
	\section{Fixed points of moduli of parabolic $\text{SL}(r,\bb{C})$-bundles}\label{section:fixed-point-locus}
	
	Fix a positive integer $r$ and a line bundle $\xi$ on $X$. Also, fix a system of weights and multiplicities 
	$(\boldsymbol{m,\,\alpha})$ along parabolic points $S\,\subset\, X$. In this section, we shall assume that 
	$\boldsymbol{\alpha}$ is a \textit{generic} system of weights, so that a parabolic vector bundle is parabolic 
	stable if it is parabolic semistable. Let $$\Malpha$$ denote the moduli space of stable parabolic vector bundles 
	of rank $r$ on $X$ with determinant $\xi$ and systems of multiplicities and weights $(\boldsymbol{m,\,\alpha})$ 
	along $S$. Let
	$$
	\Gamma\ :=\ \{L\in \Pic(X)\ \ \big\vert\, \ L^{\otimes r}\,\simeq\, \mc{O}_X\}
	$$ denote the group of $r$-torsion line bundles on $X$. For any $E_*\, \in\,
	\Malpha$ and $L\, \in\, \Gamma$, it is clear that $E_*\otimes L\, \in\, \Malpha$.
	Thus, the group $\Gamma$ acts on $\Malpha$.
	
	If one additionally assumes that the system of multiplicities $\boldsymbol{m}$ is of \textit{full-flag} type, 
	i.e., $m_{_{p,i}}\,=\,1$ for all $1\,\leq\, i\,\leq\, \ell(p)$ and every $p\,\in \,S$, then it was shown in
	\cite{BCD23} that the 
	codimension of the fixed point locus of $\Malpha$ under any non-trivial $r$-torsion line bundle $L\in \Gamma$ is 
	at least $3$. This codimension estimation will be extended for \textit{any} generic system of 
	weights $\boldsymbol{m}$.
	
	\begin{lemma}\label{lem:eigenbasis}
		Let $V$ be a finite-dimensional vector space over a field $k$ with a diagonalizable endomorphism
		$\varphi\,:\, V\,\longrightarrow \,V$ and a filtration by subspaces
		$$V\ =\ V_1\ \supsetneq\ V_2\ \supsetneq \ \cdots\ \supsetneq V_\ell \ \supsetneq\ 0$$
		such that $\varphi(V_i) = V_i$ for all $1\, \leq\, i\, \leq\, \ell$. There exists a basis $\mc{B}_i$ of $V_i$
		consisting of eigenvectors for $\varphi$, satisfying $\mc{B}_i\,\supset\, \mc{B}_{i+1}$ for all
		$1\,\leq\, i\,\leq \,\ell$.
	\end{lemma}
	
	\begin{proof}
		By standard theory of semisimple modules,  $\varphi|_{V_i}$ is diagonalizable on $V_i$ for each $i$, so
		there exists a basis $\mc{B}_{i}$ of $V_{i}$ consisting of eigenvectors. Moreover, one can write
		$V_{i-1}\,=\,W_{i-1}\oplus V_{i}$ for some subspace $W_{i-1}$ which is invariant under $\varphi$. Thus,
		one can start with a basis $\mc{B}_{\ell}$ of $V_{\ell}$ consisting of eigenvectors, next choose a direct
		summand $W_{\ell-1}$ of $V_{\ell}$ in $V_{\ell-1}$, namely $V_{\ell-1} \,=\, W_{\ell-1}\oplus V_{\ell}$, and choose an
		eigenbasis of $W_{\ell-1}$;  its union with $\mc{B}_{\ell}$ produces a basis $\mc{B}_{\ell-1}$ for $V_{\ell-1}$
		consisting of eigenvectors and satisfying $\mc{B}_{\ell-1}\,\supset\,\mc{B}_{\ell}$. Proceeding inductively, one
		can produce bases $\mc{B}_i$ of $V_i$. This proves the lemma.
	\end{proof}
	
	For any $L\,\in \,\Gamma\setminus\{\mathcal{O}_X\}$, let
	\begin{equation*}
		\MLalpha\,\ \subset\,\ \Malpha
	\end{equation*}
	be the locus of fixed points for the action of $L$ --- through tensoring --- on $\Malpha$. 
	If $d\,=\,\text{ord}(L)$, choosing a nowhere-vanishing section $s_0\, \in\, H^0(X,\, L^{\otimes d})$,
	construct the spectral curve
	\begin{align}\label{eqn:spectral-curve}
		Y\, :=\, \{v\, \in\, L\,\ \mid\,\ v^{\otimes d}\, \in\, s_0(X)\}.
	\end{align}
	The natural projection $\gamma\,:\,Y\,\longrightarrow\, X$ is an \'etale Galois covering
	with Galois group ${\mathbb Z}/d{\mathbb Z}$. The isomorphism class of this
	covering $\gamma$ does not depend on the choice of the above section $s_0$.
	
	For simplicity, let us first consider the case where $S\,=\,\{p\}$ is a singleton set.  Thus we have a
	system of weights $(\alpha_{_{p,1}}\,<\ \alpha_{_{p,2}}\ <\,\cdots\,<\  \alpha_{_{p,\ell(p)}})$ at $p\in X$. Consider the spectral curve $\gamma:Y\longrightarrow X$ from \eqref{eqn:spectral-curve}. Let
	$$\textnormal{Gal}(\gamma)\ =\ \{1, \,\mu,\,\mu^2,\,\cdots,\,\mu^{d-1}\}\,\subset \,\bb{C}^*.$$ 
	The action of $\textnormal{Gal}(\gamma)$ on $\gamma^{-1}(p)$ is via multiplication. Fix 
	an ordering on the points of $\gamma^{-1}(p)$, say 
	\begin{align}\label{eqn:fibre}
		\gamma^{-1}(p)\ =\ \{q_1,\,q_2,\,\cdots,\,q_d\},
	\end{align}
	such that $\mu^i$ acts on $\gamma^{-1}(p)$ as the cyclic permutation sending $q_j$ to $q_{j+i}$, where the 
	subscript $(j+i)$ is to be understood modulo $d$.
	
	Now, recall that the systems of multiplicities and weights at $p$ are given by
	$$\left(m_{_{p,1}},\ m_{_{p,2}},\ \cdots,\ m_{_{p,\ell(p)}}\right)\ \ \text{ and }\  \
	\left(\alpha_{_{p,1}},\ \alpha_{_{p,2}},\ \cdots ,\ \alpha_{_{p,\ell(p)}}\right)$$
	respectively. Let $\boldsymbol{P}$ denote the collection of all nonempty subsets of $\left\{\alpha_{_{p,1}}, \ \alpha_{_{p,2}}, \ \cdots, \ \alpha_{_{p,\ell(p)}}\right\}$ of cardinality \textit{at most} $\frac{r}{d}$, and let 
	\begin{equation}\label{eqn:set-of-subsets}
		\boldsymbol{T}\ \ :=\ \ \boldsymbol{P}\underset{d\ \textnormal{times}}{\times \cdots\times}\boldsymbol{P},
	\end{equation}
	where $d\,=\,\deg(\gamma)$.
	
	For each $\textbf{t}\,\in\, \textbf{T}$, we want to describe a procedure of associating a system of weights, as 
	well as a collection of systems of multiplicities along $\gamma^{-1}(p)\,=\,\{q_1,\ q_2,\ \cdots,\ q_d\}$. 
	First associate a system of weights to each $\textbf{t}\,\in \,\textbf{T}$ as follows: adopt the notation
	\begin{equation}\label{eqn:set-of-subsets-2}
		\textbf{t} \ =\ \left(\lambda_1(\textbf{t}),\ \lambda_2(\textbf{t}), \ \cdots ,\
		\lambda_d(\textbf{t})\right),
	\end{equation}
	where each $\lambda_j(\textbf{t})\neq \emptyset$ by assumption. Clearly each $\lambda_j(\textbf{t})$ can be 
	arranged into an increasing sequence. Denote by $\lambda_j(\textbf{t})$ the set of weights at $q_j$ for 
	each $1\,\leq\, j\,\leq\, d$. Thus, each $\textbf{t}\,\in\, \boldsymbol{T}$ prescribes a system of weights on 
	$\gamma^{-1}(p)$. Denote this system of weights determined by $\textbf{t}$ by
	$\boldsymbol{\alpha}_{\textbf{t}}$.
	
	Next, for each fixed $\textbf{t}\,\in\, \textbf{T}$, denote by
	\begin{align}\label{eqn:collection-of-sequences}
		A_{\textbf{t}}
	\end{align}
	be the collection of all matrices of size $(d\times \ell(p))$ with \textit{non-negative integer} entries, written in the form
	\begin{align}\label{eqn:matrix}
		\begin{pmatrix}
			n_{_{q_1,1}} & n_{_{q_1,2}} & \cdots & \cdots& n_{_{q_1,\ell(p)}} \\
			n_{_{q_2,1}} & n_{_{q_2,2}} & \cdots & \cdots& n_{_{q_2,\ell(p)}} \\
			\vdots & \vdots & \vdots& \vdots &\vdots\\
			n_{_{q_d,1}} & n_{_{q_d,2}} & \cdots & \cdots& n_{_{q_d,\ell(p)}} \\
		\end{pmatrix}
	\end{align}
	which further satisfy the following two conditions:
	\begin{enumerate}[(a)]
		\item \label{condition-on-multiplicity-1} The numbers in the $j$-th row, namely the
		sequence 
		$$\left(n_{_{q_j,1}}, \ n_{_{q_j,2}},\ \cdots,\ n_{_{q_j,\ell(p)}}\right)$$
		must satisfy the condition that for every  $1\,\leq\, k\,\leq\,\ell(p)$,
		$$(n_{_{q_j,k}}\ =\ 0)\ \ \iff\ \ (\alpha_{_{p,k}}\ \notin\ \lambda_j(\textbf{t})) 
		$$ 
		(\text{see}\ \eqref{eqn:set-of-subsets-2}); note that this makes sense because
		$\lambda_j(\textbf{t})\,\subset\,\{\alpha_{_{p,1}},\ \alpha_{_{p,2}},\ \cdots,\ \alpha_{_{p,\ell(p)}}\}$
		by construction.
		
		\item \label{condition-on-multiplicity-2} 
			the entries in each row of the matrix must add up to $\dfrac{r}{d}$, while the entries
			of the $k$-th column of it must add up to $m_{_{p,k}}$ for each $k\,\in\,[1,\,\ell(p)]$. In other words,
			\begin{align*}
				\sum_{k=1}^{\ell(p)} n_{_{q_j,k}} \ &= \ \frac{r}{d}\ \ \ \forall\ \ j\,\in\,[1,\,d]\\
				\text{and}\ \ \sum_{j=1}^{d} n_{_{q_j,k}} \ &= \ m_{_{p,k}}\ \ \ \forall\ \ k\,\in\,[1,\,\ell(p)].
			\end{align*}
		\end{enumerate}  
		
		\begin{lemma}\label{lem:system-of-multiplicities}
			Fix $\textbf{t}\in \textbf{T}$ (see \eqref{eqn:set-of-subsets}). Consider the collection of matrices satisfying the
			conditions \eqref{condition-on-multiplicity-1} and \eqref{condition-on-multiplicity-2} (the collection is denoted by $A_{\textbf{t}}$ in \eqref{eqn:collection-of-sequences}). Each such matrix gives rise to a system of multiplicities along $\gamma^{-1}(p)$ which is compatible with the system of weights $\boldsymbol{\alpha}_{\textbf{t}}$ described just before \eqref{eqn:collection-of-sequences}.
		\end{lemma}
		
		\begin{proof}
			For any matrix of the form \eqref{eqn:matrix} coming from the collection $A_{\textbf{t}}$, its rows are indexed by the points
			in $\gamma^{-1}(p)\,=\,\{q_1,\ q_2,\ \cdots,\ q_d\}$ by construction. 
			For each $j\,\in\,[1,\,d]$, discard the zero terms from the sequence $\left(n_{_{q_j,1}}\ ,n_{_{q_j,2}},\ \cdots,\ 
			n_{_{q_j,\ell(p)}}\right)$. From condition \eqref{condition-on-multiplicity-2} it follows that $n_{_{q_j,k}}$ 
			is non-zero only when $\alpha_{_{p,k}}\,\in\, \lambda_j(\textbf{t})$; in other words, $n_{_{p,k}}\neq 0$ only when $\alpha_{_{p,k}}$ appears as a weight at $q_j\in\gamma^{-1}(p)$ in the system of weights $\boldsymbol{\alpha}_{\textbf{t}}$. 
			Thus, whenever $n_{_{q_j,k}}\neq 0$, we can associate $n_{_{q_j,k}}$ as the multiplicity of the weight $\alpha_{_{p,k}}$ at $q_j\in \gamma^{-1}(p)$ (see Definition \ref{def:parabolic-bundles}). This procedure certainly
			gives rise to a system of multiplicities along $\gamma^{-1}(p)$ compatible with 
			$\boldsymbol{\alpha}_{\textbf{t}}$, as in the statement of the lemma.
		\end{proof}
		
		To aid the understanding of the above construction, we make the following picture for summarizing the above construction of multiplicities. Below, the vertical arrow denotes the map $\gamma$.
		\[\begin{array}{cc}
			{\begin{array}{c}
					\textcolor{blue}{q_1} \\ \textcolor{blue}{q_2}\\ \textcolor{blue}{\vdots} \\ \textcolor{blue}{q_d}
			\end{array}} &
			{\overbrace{
					\begin{array}{cccccccc}
						n_{_{q_1,1}} & n_{_{q_1,2}} & \cdots & \cdots& n_{_{q_1,\ell(p)}} \\
						n_{_{q_2,1}} & n_{_{q_2,2}} & \cdots & \cdots& n_{_{q_2,\ell(p)}} \\
						\vdots & \vdots & \cdots & \cdots&\vdots\\
						n_{_{q_d,1}} & n_{_{q_d,2}} & \cdots & \cdots&n_{_{q_d,\ell(p)}}
					\end{array}
				}^{\text{numbers in each row add up to}\ \frac{r}{d}\ \text{horizontally}}}\\
			&\\
			\rotatebox{-90}{$\xrightarrow{\text{\ \ \ \ \ }}$} &  \\
			&\\
			\textcolor{blue}{p} & {\underbrace{\begin{array}{cccccccc}
						m_{_{p,1}} & m_{_{p,2}} & \cdots & \cdots & m_{_{p,\ell(p)}}
				\end{array}}_{\text{each}\ k\text{-th column adds up to }\ m_{_{p,k}}}}
		\end{array}
		\]
		For each $\textbf{t}\,\in\, \textbf{T}$ consider the system of weights
		$\boldsymbol{\alpha}_{\textbf{t}}$ along $\gamma^{-1}(p)$ as described just before
		\eqref{eqn:set-of-subsets}. Consider the collection $A_{\textbf{t}}$ of all possible matrices satisfying conditions \eqref{condition-on-multiplicity-1} and \eqref{condition-on-multiplicity-2}. Let $e\,:=\,\deg(\xi)$, and for each system of multiplicities $\textbf{n}$ constructed from a matrix in $A_{\textbf{t}}$ as in Lemma \ref{lem:system-of-multiplicities}, let
		\begin{align*}
			\mc{M}^{\textbf{n},\boldsymbol{\alpha}_{\textbf{t}}}_{Y}\left(\frac{r}{d},e\right)
		\end{align*}
		denote the moduli space of parabolic stable vector bundles on $Y$ of rank $\frac{r}{d}$ and
		degree $e\,=\,\deg(\xi)$, and having the system of multiplicities and weights given by
		$(\boldsymbol{n},\,\boldsymbol{\alpha}_{\textbf{t}})$. Consider the following scheme  for each
		$\textbf{t}\,\in\, \textbf{T}$:
		\begin{equation}\label{eqn:moduli}
			\mathcal{N}^\textbf{t}_L\ :=\ \left\{F_*\,\in\, \coprod_{\boldsymbol{n}\,\in\,
				A_{\textbf{t}}}\mc{M}_Y^{\boldsymbol{n},\boldsymbol{\alpha}_{\textbf{t}}}\left(\frac{r}{d},e\right)\
			\big\vert\ \det(\gamma_*F)\,\simeq\, \xi\right\}.
		\end{equation}
		
		\begin{proposition}\label{prop:parabolic-fixed-point}
			Assume that $\boldsymbol{\alpha}$ is a generic system of weights. Consider a non-trivial $r$-torsion line bundle $L$ on $X$.
			Then there is a surjective morphism 
			\begin{equation*}
				f\,:\,\left(\coprod_{\textbf{t}\in \textbf{T}}\mathcal{N}^{\textbf{t}}_L\right)\,\longrightarrow\, \MLalpha
			\end{equation*} given by
			parabolic push-forward along the spectral curve $\gamma \,:\, Y\,\longrightarrow\, X$ associated to
			$L$ (see \eqref{eqn:spectral-curve}).
		\end{proposition}
		
		\begin{proof}
			For the sake of avoiding complicated notation, we assume that there is only one parabolic point
			$S\,=\,\{p\}$. The general case is very similar. The 
			map $f$ in the statement of the lemma sends any $F_*$
			to the parabolic push-forward of $F_*$ constructed in Section \ref{parabolic pushforward}.
			
			It will be shown that $f(F_*)$ is parabolic semistable. Denote $E_*\,:= \,f(F_*)$. Then 
			$$\gamma^* E_* \cong\ \bigoplus_{\sigma\in 
				\textnormal{Gal}(\gamma)} \sigma^* F_*
			$$ 
			(\cite[p.~68]{Ses}), where $\sigma^* F_*$ has the obvious parabolic 
			structure coming from $F_*$. Clearly $\text{par}\mu(\sigma^*F_*)\,=\,\text{par}\mu(F_*)$ for all $\sigma$.
			Consequently, $\gamma^* E_*$ is a direct sum of parabolic stable bundles of same parabolic slope, which 
			immediately implies that $\gamma^* E_*$ is parabolic semistable. Thus $E_*$ must be parabolic 
			semistable as well, since any sub-bundle $E'_*\,\subset\, E_*$ with strictly larger parabolic 
			slope would give rise to a subbundle $\gamma^* E'_*\,\subset \,\gamma^* E_*$ of strictly 
			larger parabolic slope, contradicting the parabolic semistability of $\gamma^* E_*$.
			Therefore, $f(F_*)$ is parabolic semistable.
			
			As $\boldsymbol{\alpha}$ is a generic system of weights, it now
			follows that $E_* \,=\, f(F_*)$ is parabolic stable.
			
			It will be shown that $\textnormal{Im}(f)\,\subseteq\,\MLalpha$. There is a tautological 
			trivialization of $\gamma^*L$ over $Y$, which induces an isomorphism 
			\begin{equation}\label{tautological trivialization}
				\theta\ :\ \mathcal{O}_{Y}\ \overset{\simeq}{\longrightarrow}\ \gamma^*L.
			\end{equation}  
			This induces an isomorphism
			\begin{equation*}
				{\rm Id}_F\otimes \theta \ :\ F\ \xrightarrow{\,\,\,\simeq\,\,\,}\ F\otimes \gamma^*L.
			\end{equation*}
			Consider
			$$
			\psi\, :=\, \gamma_*({\rm Id}_F\otimes\theta)\, :\, \gamma_*F \, \xrightarrow{\,\,\,\simeq\,\,\,}\,
			\gamma_*(F\otimes \gamma^*L)\,=\, (\gamma_*F)\otimes L\,=\, E\otimes L;
			$$
			the above $\gamma_*(F\otimes \gamma^*L)\,=\, (\gamma_*F)\otimes L$ is given by the projection formula.
			Now $E_p\,=\, \bigoplus_{i=1}^{d}F_{q_i}$ (see \eqref{eqn:fibre}), and the map $\psi_p\,:\,E_p
			\,\longrightarrow\, E_p\otimes L_p$ on the fiber takes $F_{q_i}$ to $F_{q_i}\otimes L_p$, which clearly implies that
			$\psi_p$ preserves the filtration induced on $E_p$. Thus we have $\text{Im}(f)\,\subseteq\, \MLalpha.$
			
			To show that the map $f$ is surjective, take any $E_*\,\in\, \MLalpha$, so there is an isomorphism
			of parabolic bundles 
			\[\varphi_*\ :\ E_*\ \xrightarrow{\,\,\,\simeq\,\,\,}\ E_*\otimes L.\] 
			Since $E_*$ is parabolic stable, it is simple, and hence any parabolic endomorphism of $E_*$ is 
			a constant scalar multiplication. As a consequence, any two parabolic isomorphisms from 
			$E_*$ to $E_*\otimes L$ will differ by a constant scalar multiplication. Thus, we can re-scale $\varphi_*$ 
			by multiplying with a nonzero scalar, so that the $d$--fold composition
			\[\underset{\text{d-times}}{\varphi_*\,\circ\cdots\circ\,\varphi_*}\ :\ E_*
			\ \longrightarrow\ E_*\otimes L^d\]
			coincides with $\textnormal{Id}_{E_*}\otimes s_0$, where $s_0$ is the
			nowhere vanishing section of $L^d$ used in the construction of $Y$ (see
			\eqref{eqn:fibre}).
			It follows from the proof of \cite[Lemma 2.1]{BH10} that there exists a vector bundle $F$ on $Y$ of rank
			$\frac{r}{d}$ with $\gamma_*(F)\,\cong\, E$; the argument from \cite{BH10} will be briefly
			recalled. Consider the pull-back $\gamma^*\varphi$, and compose it with the tautological trivialization of
			$\gamma^*L$ to get a homomorphism 
			\[\phi\ :\ \gamma^*E\ \longrightarrow\ \gamma^*E.\] 
			Since $Y$ is irreducible, the characteristic polynomial of $\phi_y$ remains unchanged as
			$y\,\in \,Y$ moves. Consequently, $\gamma^*E$ decomposes into eigenspace
			sub-bundles. If $F$ is an eigenspace sub-bundle of $\gamma^*E$, then
			$\gamma_* F\,\cong\, E$. Moreover, the decomposition $\gamma^* E\,=\,\bigoplus_{i=1}^{d}(\mu^i)^*F$ is precisely
			the decomposition of $\gamma^* E$ into eigenspace sub-bundles.
			
			Our next task is to produce a parabolic structure on $F$ along $\gamma^{-1}(p)$ so that the parabolic push-forward $\gamma_*(F_*)$ (in the sense of \S~\ref{parabolic pushforward}) coincides with $E_*$. First, 
			recall the description of $\theta$ in (\ref{tautological trivialization}), and notice that for any choice of
			$q\,\in\, \gamma^{-1}(p)$, the map $\phi_q$ is precisely the composition of maps
			\begin{align}\label{phi_q}
				(\gamma^*E)_q=E_p\xrightarrow{\varphi_p} E_p\otimes L_p = (\gamma^*E)_q\otimes (\gamma^*L)_q\xrightarrow{Id\otimes(\theta_q)^{-1}}(\gamma^*E)_q,\,
			\end{align}
			where $\theta_q\,:\, \bb{C}\,\longrightarrow\, (\gamma^*L)_q\,=\,L_p$ is defined by
			$z\,\longmapsto\, z\cdot q$. Thus, if
			\begin{align*}
				E_p\,=\,E_{_{p,1}}\,\supsetneq \,E_{_{p,2}}\,\supsetneq\,\cdots\,\supsetneq\,E_{_{p,\ell(p)}}\,\supsetneq\,0
			\end{align*}
			is the given parabolic filtration of the fiber $E_p$, then as $\varphi_*$ is a parabolic isomorphism, the following
			holds:
			\begin{align}
				\forall\ j\in[1,\,\ell(p)],\,\,\,\left(\varphi_p(E_{_{p,j}})\,=\,E_{_{p,j}}\otimes L_p\right)\,\implies\,\left(\phi_q(E_{_{p,j}})\,=\,E_{_{p,j}}\right)\,\,\,\,[\text{from}\,(\ref{phi_q})].
			\end{align}
			We have thus obtained a diagonalizable linear map $\phi_q\,:\,E_p\,\longrightarrow\, E_p$
			with the property that the filtration of $E_p$ given by
			\begin{equation*}
				E_p\,=\,E_{_{p,1}}
				\,\supsetneq\,\cdots\,\supsetneq\, E_{_{p,\ell(p)}}\,\supsetneq\,0
			\end{equation*}
			is preserved by $\phi_q$. Now Lemma \ref{lem:eigenbasis}
			guarantees the existence of a basis $\mc{B}_j$ of $E_{_{p,j}}$ consisting of eigenvectors for each
			$i$, satisfying  $\mc{B}_{j}\,\supset\,\mc{B}_{j+1}$ for all $1\,\leq\, j\,\leq\, \ell(p)$. 
			
			From these data, a weighted filtration is prescribed on the fiber $F_{q_j}$ of the vector bundle $F$ for each
			$q_j\,\in\, \gamma^{-1}(p)$ in the following manner: for each $k\,\in\,[1,\,\ell(p)]$, whenever
			$\mc{B}_{k}\cap F_{q_{_{j}}}\,\neq\,\emptyset$, define the following subspace of $F_{_{q_j}}$:
			$$F_{_{q_{_{j}},k}}\ :=\ \langle\mc{B}_{k}\cap F_{q_{_{j}}}\rangle\ \ \ (\text{the subspace generated by the
				intersection}).$$
			Now, for any two such numbers $k\,<\,k'$, we have
			$\mc{B}_k\,\supset\, \mc{B}_{k'},$ which implies that $F_{_{q_j,k}}\,\supset\, F_{_{q_j,k'}}.$
			This immediately produces a weighted filtration of $F_{q_{j}}$ by assigning the
			weight $\alpha_{_{p,k}}$ to   $F_{_{q_{j},k}}$ whenever $\mc{B}_k\cap F_{_{q_j}}\,\neq\,\emptyset$.
			
			By applying this process for all $1\,\leq\, j\,\leq\, \ell(p)$, a parabolic vector bundle $F_*$ on $Y$
			is constructed, for which (see \S~\ref{parabolic pushforward})
			$$\gamma_* F_* \ =\ E_* .$$
			It is not hard to see that the system of multiplicities associated to such filtration do come from a matrix as in \eqref{eqn:matrix}.

			Note that $F_*$ must be parabolic stable, because if $F'\,\subset\, F$ is any sub-bundle
			such that $\text{par}\mu(F'_*)\,\geq \,\text{par}\mu(F_*)$ (see Definition \ref{def:parabolic-degree-and-slope}), then the equalities
			\[\text{par}\deg(\gamma_*(F'_*))\,=\,\text{par}\deg(F'_*)\ \ \text{ and }\ \ 
			\text{rank}(\gamma_*(F')) \,=\, m\cdot \text{rank}(F')\]
			would imply that $\gamma_*(F')\,\subset\, E$ violates the parabolic stability condition
			of $E_*$. Thus $F_*\,\in\, \mathcal{N}^\textbf{t}_L$ (see \eqref{eqn:moduli}) for
			some $\textbf{t}\,\in\, \boldsymbol{T}$ (see \eqref{eqn:set-of-subsets}).   and the proposition follows.
		\end{proof}
		
		\begin{remark}\label{rem:multipli-parabolic-points}
			If the cardinality of the set of parabolic points $S\,\subset\, X$ is greater than $1$, we repeat this
			construction for each of the parabolic points in $S$. Say, for example, $S\,=\,\{p_1,\ p_2,\ \cdots,\ p_n\}$;
			then for each $1\,\leq\, i\,\leq\, n$, define the set $\boldsymbol{T}_i$ analogous to $\boldsymbol{T}$ in
			\eqref{eqn:set-of-subsets}, and replace $\boldsymbol{T}$ by $\boldsymbol{T}_1\times\boldsymbol{T}_2
			\times\cdots\boldsymbol{T}_n$. It is not difficult to see that a similar argument as in Proposition
			\ref{prop:parabolic-fixed-point} will prove the same result, albeit the notation will get more cumbersome.
		\end{remark}
		
		\begin{corollary}\label{cor:codimension-estimate}
			Let $\boldsymbol{\alpha}$ be a generic system of weights. The codimension of the closed subscheme
			\begin{equation*}
				Z_{\boldsymbol{\alpha}}\ :=\ \underset{L\in\Gamma\setminus\{\mathcal{O}_X\}}{\bigcup}\MLalpha\
				\subset\ \Malpha
			\end{equation*}
			is at least three.
		\end{corollary}
		
		\begin{proof}
			Recall the filtration data along the parabolic points $S$ as described in \eqref{eqn:filtration-data}. Now, for
			any parabolic point $p\,\in\, S$, the flag variety of all filtrations of $\bb{C}^r$ of flag type
			$(m_{_{p,1}} ,\, m_{_{p,2}},\, \cdots ,\, m_{_{p,\ell(p)}})$ (where $m_{_{p,j}}$'s are as in Definition \ref{def:parabolic-bundles}) is of dimension
			\begin{equation}\label{eqn:flag-variety-dimension}
				\sum_{i=1}^{\ell(p)-1}m_{_{p,i}}\left(m_{_{p,i+1}}\ +\  \cdots\ +m_{_{p,\ell(p)}}\right)\ \ 
				(\text{see, e.g. \cite[2.1]{B04}}.)
			\end{equation}
			Now, from \cite{MS80} we get dimension of the moduli space $\Malpha$ to be
			$$\dim\left(\Malpha\right)\ = \ (r^2-1)(g-1)+\sum_{p\in S}\sum_{i=1}^{\ell(p)-1}m_{_{p,i}}\left(m_{_{p,i+1}}\ 
			+\  \cdots\ +m_{_{p,\ell(p)}}\right).$$
			Moreover, writing $\gamma^{-1}(p)\,=\,\{q_{_{p,1}},\ \cdots, \ q_{_{p,d}}\}$ for each $p\,\in\, S$, by
			condition \eqref{condition-on-multiplicity-2} on the collection $A_{\textbf{t}}$
			(see \eqref{eqn:collection-of-sequences}),
			\begin{equation*}
				\sum_{j=1}^{d} n_{_{q_{_{p,j}},i}}\ =\ m_{_{p,i}}\ \ \text{ for each}\ \  1\leq i\leq \ell(p)\ \ \text{and}\ \ p\in S.
			\end{equation*} 
			Thus, for each $\textbf{t}\,\in\, \boldsymbol{T}$ (see \eqref{eqn:set-of-subsets}), using the description
			of $\mc{N}^\textbf{t}_L$ in \eqref{eqn:moduli} and by allowing some of the $n_{_{q_{p,j},i}}$'s to be zero
			in the following equation,
			\begin{align*}
				\dim(\mc{N}^\textbf{t}_L)\ &=\ \left(\frac{r^2}{d^2}(\text{genus}(Y)-1)+1\right)-g\ +\
				\sum_{p\in S}\left(\sum_{j=1}^{d} \sum_{i=1}^{\ell(p)-1} n_{_{{q_{_{p,j}},i}}}\left(n_{_{q_{_{p,j}},i+1}}\
				+\ \cdots\ +\  n_{_{q_{_{p,j}},\ell(p)}}\right)\right)\nonumber\\
				& (\text{by \eqref{eqn:flag-variety-dimension}})\nonumber\\
				&= \ (g-1)(\frac{r^2}{d}-1)\ +\ \sum_{p\in S}\left(\sum_{j=1}^{d} \sum_{i=1}^{\ell(p)-1}
				n_{_{{q_{_{p,j}},i}}}\left(n_{_{q_{_{p,j}},i+1}}\ +\ \cdots\ +\  n_{_{q_{_{p,j}},\ell(p)}}\right)\right)\nonumber\\
				& [\text{because}\ \, \text{genus}(Y)\,=\,m(g-1)+1]\nonumber\\
				&\leq \ (g-1)(\frac{r^2}{d}-1)\ +\ \sum_{p\in S}\left(\sum_{j=1}^{d} \sum_{i=1}^{\ell(p)-1} n_{_{q_{_{p,j}},i}}
				\left(m_{_{p,i+1}}\ +\ \cdots\ +\  m_{_{p,\ell(p)}}\right)\right)\nonumber\\
				& [\text{because}\ \, n_{_{q_{_{p,j}},i}}\, \leq\, m_{_{p,i}}\ \, \forall \ \,i,\,j]\nonumber
			\end{align*}
			
			It follows that, for each $\textbf{t}\,\in\, \boldsymbol{T}$,
			\begin{align*}
				&\dim\left(\Malpha\right)- \dim\left(\mc{N}^\textbf{t}_L\right)\nonumber\\
				&\geq\ r^2(g-1)\left(1-\frac{1}{d}\right)\ +\ 
				\sum_{p\in S}\left( \sum_{i=1}^{\ell(p)-1}\left(m_{_{p,i+1}}\ +\ \cdots\ +\
				m_{_{p,\ell(p)}}\right) \left(m_{_{p,i}}-\sum_{j=1}^{d}n_{_{q_{_j},i}}\right)\right)\nonumber\\
				&\geq\ r^2(g-1)\left(1-\frac{1}{d}\right)\ \ \ \ \ \left[\text{because}\ \,m_{_{p,i}}\,=
				\, \sum_{j=1}^{d}n_{_{q_{_j}},i}\ \ \text{by condition}\ \eqref{condition-on-multiplicity-2}\right]\nonumber\\
				&\geq\ 3 .
			\end{align*}
			The result now follows from Proposition \ref{prop:parabolic-fixed-point} and Remark \ref{rem:multipli-parabolic-points}.
		\end{proof}
		
		\section{Brauer group of the moduli stack of parabolic
			$\text{PGL}(r,\bb{C})$-bundles}\label{section:brauer-group-moduli-stack}
		
		As before, fix a finite set of points $S\,\subset\, X$, a positive integer $r$ and line bundle $\xi$ on $X$.
		
		\begin{definition}
			A parabolic stable $\textnormal{PGL}(r,\bb{C})$-bundle is an equivalence class of parabolic stable vector bundles, where two parabolic stable vector bundles $E_*$ and $E'_*$ 
			are considered equivalent if there exists a line bundle $\mc{L}$ such that $E'_*\,\simeq\, E_*\otimes \mc{L}$ as parabolic vector bundles.
		\end{definition}
		
		Let $\delta\,\in\,[0,\,r-1]$ be the unique integer satisfying the condition
		$\deg(\xi)\ \equiv\ \delta\ (\text{mod}\ r)$. The coarse moduli space of parabolic stable
		$\text{PGL}(r,\bb{C})$--bundles on $X$ of topological type $\delta$ and systems of multiplicities and
		weights $(\boldsymbol{m,\,\alpha})$ will be denoted by $\mc{N}^{\boldsymbol{m,\alpha}}_X(r,\delta)$.
		Also, denote by $\mf{N}^{\boldsymbol{m,\alpha}}_X(r,\delta)$ the moduli \textit{stack} of parabolic stable
		$\text{PGL}(r,\bb{C})$--bundles.
		
		\begin{theorem}\label{thm:brauer-group-of-moduli-stack}
			Let $\boldsymbol{\alpha}$ be a generic system of weights. Let $\Nalpha^{sm}$ denote the smooth locus of $\Nalpha$. Then,
			$$\Br\left(\mf{N}^{\boldsymbol{m,\alpha}}_X(r,\delta)\right)\,\ \simeq\,\
			\Br\left(\Nalpha^{sm}\right).$$
		\end{theorem}
		
		\begin{proof}
			Recall that the group $\Gamma$ of $r$-torsion line bundles on $X$ act on the moduli space $\Malpha$ of parabolic 
			stable vector bundles. Through this action, $\Nalpha$ is naturally isomorphic to the quotient variety 
			$\Malpha/\Gamma$ (see \cite[p.~196, \textbf{6}]{BLS98}).
			
			Let $Z_{\boldsymbol{\alpha}}$ be as in Corollary \ref{cor:codimension-estimate}, and denote
			$\mc{U}\,:=\,\Malpha\setminus Z_{\boldsymbol{\alpha}}$. Consider the following diagram:
			\begin{align}\label{diagram:diag-4}
				\xymatrix{[\mc{U}/\Gamma]\  \ar@{^{(}->}[r] \ar[d]\  &\  [\Malpha/\Gamma]\ \ar@{=}[r] \ar[d] \ &\  \mf{N}^{\boldsymbol{m,\alpha}}_X(r,\delta) \ar[d] \\
					\mc{U}/\Gamma \ \ar@{^{(}->}[r]\  &\  \Malpha/\Gamma \ \ar@{=}[r]\  &\  \Nalpha
				}
			\end{align}
			where each of the top horizontal arrows correspond to an inclusion of open sub-stack, while the bottom
			horizontal arrows correspond to inclusions of open sub-schemes; the
			vertical maps correspond to morphisms to coarse moduli spaces.
			Since taking quotient by the finite group $\Gamma$ is a finite morphism and hence codimension-preserving, it follows from Corollary \ref{cor:codimension-estimate} that the complement of $\mc{U}/\Gamma$ in $\left(\mc{N}^{\boldsymbol{m,\alpha}}_X(r,\delta)\right)^{sm}$ is of codimension at least $3$. For a similar reason, the complement of the open sub-stack $[\mc{U}/\Gamma]$ in $\mf{N}^{\boldsymbol{m,\alpha}}_X(r,\delta)$ is of codimension at least $3$ as well. As $\mf{N}^{\boldsymbol{m,\alpha}}_X(r,\delta)$ is a Deligne-Mumford stack, it follows that
			\begin{equation*}
				\Br\left([\mc{U}/\Gamma]\right)\,\ \simeq\, \ \Br\left(\mf{N}^{\boldsymbol{m,\alpha}}_X(r,\delta)\right)
			\end{equation*}
			(see \cite[Proposition 4.2]{BCD23}). Now, as the action of $\Gamma$ on $\mc{U}$ is free, the left-most
			vertical arrow in the diagram \eqref{diagram:diag-4} is an isomorphism. Hence
			$$\Br\left([\mc{U}/\Gamma]\right)\ \simeq\
			\Br\left(\mc{U}/\Gamma\right).$$
			Also, $\Malpha$ is a smooth variety, and
			hence $\mc{U}$ is smooth. As $\Gamma$ acts freely on $\mc{U}$, the quotient $\mc{U}/\Gamma$ is also
			smooth. The complement of $\mc{U}/\Gamma$ in $\left(\Nalpha\right)^{sm}$
			clearly has codimension at least $3$ as well. Thus we have
			$$\Br\left(\mf{N}^{\boldsymbol{m,\alpha}}_X(r,\delta)\right)\ \simeq \ Br\left([\mc{U}/\Gamma]\right)\ \simeq\
			\Br\left(\mc{U}/\Gamma\right)\ \simeq\  \Br\left(\mc{N}^{\boldsymbol{m,\alpha}}_X(r,\delta)^{sm}\right),$$  
			where the last isomorphism follows from \textnormal{\cite[Theorem 1.1]{C19}}. This
			completes the proof.
		\end{proof}
		
		\section{Brauer group of moduli stack of quasi-parabolic 
			$G$--bundles}\label{section:brauer-group-quasi-parabolic-bundles}
		
		Let $G$ be a simple and simply connected affine algebraic group over $\bb{C}$. Let $X$ be a smooth projective
		curve over $\bb{C}$ of genus $g$, with $g\geq 2$. We aim to compute the Brauer group of the moduli stack of
		quasi-parabolic $G$--bundles on $X$. To be more precise, following notation of \cite{LS97}, let 
		$$S\ =\ \{p_1,\,\cdots,\,p_n\}$$ 
		be the set of parabolic points on $X$, and fix a parabolic subgroup $P_i\, \subset\,G$
		for each $p_i$. Let $\underline{p}$ and $\underline{P}$ denote the sequences $(p_1,\ p_2,\ \cdots,\ p_n)$
		and $(P_1,\ P_2,\ \cdots,\ P_n)$ respectively. For a principal $G$--bundle $E$ on $X$, and each $i\,\in\,[1,\,n]$,
		let $E(G/P_i)\, \longrightarrow\, X$ denote the fiber bundle associated to $E$ for the left-translation
		action of $G$ on $G/P_i$.
		
		\begin{definition}[{\cite[(8.4)]{LS97}}]\label{def:quasi-parabolic-bundle}
			A \textit{quasi-parabolic principal $G$-bundle} of type $\underline{P}$ is a principal $G$-bundle $E$ on
			$X$, together with an element $F_i$ in the fiber $E(G/P_i)_{p_i}$ of $E(G/P_i)$ over $p_i$.
		\end{definition}
		
		\begin{remark}
			Consider the case where $G\,=\,\text{SL}(r,\bb{C})$. A principal $\text{SL}(r,\bb{C})$--bundle corresponds to a 
			vector bundle $E$ of rank $r$ with $\det(E)\,\simeq\,\mc{O}_X$. Parabolic subgroups of 
			$\text{SL}(r,\bb{C})$ correspond to block-upper triangular matrices. So 
			the associated fiber bundle $E\left(G/P_i\right)$ is a flag 
			bundle for $E$ of flag type determined by $P_i$. Choosing of an element $F_i$ in the 
			fiber $E\left(G/P_i\right)_{p_i}$ is equivalent to fixing a filtration of the fiber $E_{_{p_i}}$ whose 
			multiplicities (see Definition \ref{def:parabolic-bundles}) are given by the
			dimensions of the blocks for the subgroup 
			$P_i$ of block-upper triangular matrices. This way, Definition \ref{def:quasi-parabolic-bundle} connects to 
			earlier Definition \ref{def:parabolic-bundles} in the case when $G=\text{SL}(r,\bb{C})$.
		\end{remark}
		
		Let $\mc{M}_G^{par}\left(\underline{p},\ \underline{P}\right)$ denote the moduli stack of quasi-parabolic 
		principal $G$-bundles on $X$ of type $\underline{P}$ (Definition \ref{def:quasi-parabolic-bundle}). There is a 
		very useful result known as the ``uniformization theorem'', which expresses the stack $\mgpar(\underline{p},\ 
		\underline{P})$ as a quotient stack of a certain scheme under the action of a group scheme, which is
		briefly recalled below (see \cite{LS97} for more details).
		
		Consider the ind-Grassmannian $\mc{Q}_G$, which is a direct limit of a sequence of integral projective varieties. Define
		$$\mc{Q}_G^{par} \left(\underline{p},\ \underline{P}\right):= \mc{Q}_G\times \prod_{i=1}^{n}G/{P_i}.$$
		Fix a base-point $p_0\,\in\, X\setminus S$. The ind-group scheme $L_X(G)
		\,:= \,G(\mc{O}(X\setminus \{p_0\}))$ acts on $\mc{Q}_G^{par}\left(\underline{p},\ \underline{P}\right)$ on the
		left. Consider the quotient stack $\left[L_X(G)\backslash\mc{Q}_G^{par}\left(\underline{p},\ 
		\underline{P}\right)\right]$.
		
		\begin{theorem}[{\cite[Theorem 8.5]{LS97}}]\label{thm:uniformization}
			There is a canonical isomorphism of stacks which is locally trivial in the \'etale topology:
			\begin{align}
				\left[L_X(G)\backslash\mc{Q}_G^{par}\left(\underline{p},\ \underline{P}\right)\right]\,\ \simeq\, \ 
				\mgpar\left(\underline{p},\ \underline{P}\right).
			\end{align}
		\end{theorem}
		
		For convenience, $\mgpar\left(\underline{p},\ \underline{P}\right)$ and $\mc{Q}^{par}_G\left(\underline{p},
		\ \underline{P}\right)$ will be denoted simply by $\mc{M}_G^{par}$ and $\mc{Q}^{par}_G$ respectively.
		Also, denote
		$$F\ \, :=\ \, \prod_{i=1}^{n}G/P_i .$$  
		
		For an algebraic stack $\mathscr{X}$ over $\bb{C}$, let $H_B^*\left(\mathscr{X},\ \bb{Z}\right)$
		denote its singular cohomology. The following lemma describes the singular cohomology of the stack $\mgpar$.
		
		\begin{lemma}\label{lem:singular-cohomology}
			The following holds:
			$$H_B^*\left(\mgpar,\ \bb{Z}\right)\ \ \simeq\ \ H_B^*\left(\mc{M}_G,\ \bb{Z}\right)\otimes
			H_B^*\left(F,\ \bb{Z}\right)
			$$
		\end{lemma}
		
		\begin{proof}
			We have the following Cartesian diagram
			\begin{align}\label{diag:cartesian-diagram-1}
				\xymatrix{
					\mc{Q}_G\times F=\mc{Q}_G^{par} \ar[r] \ar[d] & \mgpar \ar[d]\\
					\mc{Q}_G \ar[r] &  \mc{M}_G
				}
			\end{align}
			where both the horizontal arrows are principal $L_X(G)$--bundles, and the left vertical arrow is the
			first projection. Denote by $\mc{Q}_{G,\bullet}$ the simplicial space of fibered powers of $\mc{Q}_G$ over
			$\mc{M}_G$ (see \cite[p.~21, \S~4]{Te98}); more precisely, for each $n\in\bb{Z}_{\geq 0}$, we have 
			$$\mc{Q}_{G,n}\ :=\ \mc{Q}_G\ \underbrace{\underset{\mc{M}_G}{\times}\ \mc{Q}_G\ \underset{\mc{M}_G}{\times}\ \cdots\ \underset{\mc{M}_G}{\times}}_{n\ times}\ \mc{Q}_G .$$ 
			Using the fact that $\mc{Q}_G\longrightarrow\mc{M}_G$ is a principal $L_X(G)$--bundle, it can be shown that 
			\begin{align}\label{eqn:cohomology}
				H^*_B\left(\mc{M}_G,\ \bb{Z}\right)\ \ \simeq\ \ \bb{H}^*_B\left(\mc{Q}_{G,\bullet},\ \bb{Z}\right);
			\end{align}
			see the proof of \cite[Proposition 4.1]{Te98}.
			Now, if $\mc{Q}_{G,\bullet}^{par}$ analogously denotes the simplicial space of fibered powers of $\mc{Q}_G^{par}$ over $\mgpar$, we observe that for each $n\geq 0$,
			\begin{align*}
				\mc{Q}_{G, n}^{par} &\ =\  \mc{Q}_G^{par}\ \underbrace{\underset{\mgpar}{\times}\ \mc{Q}^{par}_G\ \underset{\mgpar}{\times}\ \cdots\ \underset{\mgpar}{\times}}_{n\ times}\ \mc{Q}^{par}_G \\
				&\ \simeq\ (\mc{Q}_G\ \underbrace{\underset{\mc{M}_G}{\times}\ \mc{Q}_G\ \underset{\mc{M}_G}{\times}\ \cdots\ \underset{\mc{M}_G}{\times}}_{n\ times}\ \mc{Q}_G)\times F \ \ (\text{using the Cartesian diagram \eqref{diag:cartesian-diagram-1}})\\
				&\ \simeq\ \mc{Q}_{G,n}\times F.
			\end{align*}
			In other words, $\mc{Q}_{G,\bullet}^{par}\,\simeq\, \mc{Q}_{G, \bullet}\times F$ as simplicial spaces, and
			consequently by Leray--Hirsch theorem it follows that 
			\begin{align}\label{eqn:eq-2}
				\bb{H}^*_B\left(\mc{Q}_{G,\bullet}^{par},\ \bb{Z}\right)\ \simeq\ \bb{H}_B^*\left(\mc{Q}_{G,\bullet},
				\ \bb{Z}\right)\otimes H^*_B\left(F,\ \bb{Z}\right).
			\end{align}
			Now, since $\mc{Q}_G^{par}\,\longrightarrow\,\mgpar$ is a principal $L_X(G)$--bundle, as in \eqref{eqn:cohomology} we again get 
			$$H^*_B\left(\mgpar, \ \bb{Z}\right)\ \simeq\ \bb{H}^*_B\left(\mc{Q}_{G,\bullet}^{par},\ \bb{Z}\right).$$
			Using \eqref{eqn:eq-2}, we conclude that
			\begin{equation*}
				H^*_B\left(\mgpar, \ \bb{Z}\right)\ \simeq\ \bb{H}_B^*\left(\mc{Q}_{G,\bullet}, \ \bb{Z}\right)\otimes H^*_B\left(F,\ \bb{Z}\right)\ \underset{\eqref{eqn:cohomology}}{\simeq}\ H^*_B\left(\mc{M}_G,\ \bb{Z}\right)\otimes H^*_B\left(F,\ \bb{Z}\right).
			\end{equation*}
			This proves the lemma.
		\end{proof}
		
		The following result helps in comparing the Brauer group of an algebraic stack with certain Betti cohomology
		group, under certain assumptions on the stack.
		
		\begin{proposition}[{\cite[Proposition 2.1]{BH13}}]\label{prop:two-conditions}
			Let $Y$ be an algebraic stack satisfying the following two properties:
			\begin{enumerate}[(1)]
				\item Any class in $H_B^2(Y,\ \bb{Z})$ is represented by a holomorphic line bundle on $Y$.
				\item Each holomorphic line bundle on $Y$ admits an algebraic structure.
			\end{enumerate}
			Then there are isomorphisms
			\begin{equation*}
				\Br\left(Y\right)\ \simeq\ H^2(Y, \ \mc{O}^*_{Y,an})_{torsion}\ \simeq\ H_B^3(Y,\ \bb{Z})_{torsion}.
			\end{equation*}
		\end{proposition}
		
		\begin{lemma}\label{lem:brauer-group-of-qgpar}
			Let $\mc{O}^{^{an}}_{\mc{Q}^{par}_G}$ denote the sheaf of analytic functions on $\mc{Q}^{par}_G$. Then
			$$\Br\left(\mc{Q}^{par}_G\right)\ \simeq\ H^2\left(\mc{Q}_G^{par},\ 
			(\mc{O}^{^{an}}_{\mc{Q}^{par}_G})^*\right)_{torsion}\ =\ 0.$$
		\end{lemma}
		
		\begin{proof}
			By definition, $\mc{Q}_G^{par} \,=\, \mc{Q}_G\times F$, where $F\,:=\, \prod_{i=1}^{n}G/P_i$. It is known that $\mc{Q}_G$ is the direct limit of certain projective varieties $\{Q_w\}$, where $w$ varies
			in the affine Weyl group (see the proof of \cite[Proposition 3.1]{BH13} for details). Thus $\mc{Q}_G^{par}$ is the direct
			limit of the projective varieties $\{Q_w\times F\}$. Also,
			\begin{align*}
				H^1\left(Q_w,\ \bb{Z}\right)\ =\ H^3\left(Q_w,\ \bb{Z}\right)\ \ =\ \ 0
			\end{align*}
			(see the proof of \cite[Proposition 3.1]{BH13}). Now, since $F$ is a product of flag varieties, its singular
			cohomology groups exist only in even degrees. Therefore, the K\"unneth formula gives the following: 
			\begin{align*}
				H^1_B\left(Q_w\times F,\ \bb{Z}\right)\ &=\ H^3_B\left(Q_w\times F,\ \bb{Z}\right)\ =\ 0,\\
				\text{and}\ \ H^2_B\left(Q_w\times F,\ \bb{Z}\right)\ &\simeq \ H^2_B(Q_w,\ \bb{Z})\ \oplus\ H^2_B(F,\ \bb{Z}).
			\end{align*}
			As $\mc{Q}_G^{par}$ is the direct limit of the varieties
			$\{Q_w\times\ F\}$, it follows that
			\begin{align}\label{eqn:cohomology-vanishing}
				H^1_B\left(\mc{Q}_G^{par}, \ \bb{Z}\right)\ &=\ H^3_B\left(\mc{Q}_G^{par}, \ \bb{Z}\right)\ =\ 0,\\
				\text{and}\ \ H^2_B(\mc{Q}_G^{par}, \ \bb{Z})\ &\simeq\  H^2_B(\mc{Q}_G,\ \bb{Z})\ \oplus\ H^2_B(F,\ \bb{Z})\label{eqn:cohomology-of-limit}.
			\end{align}
			Let $\mc{O}_{G/P_i}$ denote the sheaf of regular functions on $G/P_i$. Since each $G/P_i$ is a rational variety, it follows that $H^1\left(G/P_i,\ \mc{O}_{G/P_i}\right)=0$.  Using \text{\cite[III Ex. 12.6]{Ha}} together with \cite[p. 157, Lemma 2.2]{KN97}, it follows that 
			\begin{align}\label{eqn:picard-group-of-limit}
				\Pic\left(\mc{Q}^{par}_G\right)\ =\  \Pic\left(\mc{Q}_G\right)\oplus \bigoplus_{i=1}^{n}\Pic\left(G/P_i\right). 
			\end{align}
			We now aim to invoke Proposition \ref{prop:two-conditions} for the stack $\mc{Q}^{par}_G$. It was argued in the
			proof of \cite[Proposition 3.1]{BH13} that $\mc{Q}_G$ satisfies the two assumptions  of Proposition
			\ref{prop:two-conditions}. The same is true for the variety $F$ as well. Combining \eqref{eqn:cohomology-of-limit}
			and \eqref{eqn:picard-group-of-limit}, we can immediately conclude that $\mc{Q}_G^{par}$ also satisfies the
			two assumptions of Proposition \ref{prop:two-conditions}. Finally, it follows from \eqref{eqn:cohomology-vanishing}
			and the comparison isomorphisms in Proposition \ref{prop:two-conditions} that
			$$\Br\left(\mc{Q}_G^{par}\right)\ \ \simeq\ \ H^2\left(\mc{Q}_G^{par},\ 
			(\mc{O}^{^{an}}_{\mc{Q}^{par}_G})^*\right)_{torsion}\ \ 
			\simeq\ \ H^3_B\left(\mc{Q}^{par}_G,\ \bb{Z}\right)\ \ =\ \ 0.$$ This proves the lemma.
		\end{proof}
		
		\begin{proposition}\label{prop:analytic-brauer-group-vanishing}
			Let $G$ be a simple and simply connected complex affine algebraic group. Let
			$\mc{O}^{^{an}}_{\mgpar}$\ \ denote the sheaf of analytic functions on $\mgpar$. The analytic Brauer group of $\mgpar$ vanishes, namely $$H^2\left(\mgpar,\ (\mc{O}^{^{an}}_{\mgpar})^*\right)_{torsion}\ =\ 0.$$
		\end{proposition}
		
		\begin{proof}
			Let $\mc{O}^{^{an}}_{\mc{Q}_G^{par}}$ denote the sheaf of analytic functions on $\mc{Q}^{par}_G$. Let
			$BL_X(G)$ denote the classifying space for $L_X(G)$. Using the fact that the quotient morphism $\mc{Q}_G^{par}\longrightarrow \mgpar$ is a principal $L_X(G)$--bundle, the Leray spectral sequence gives rise to the following long exact sequence in analytic topology (cf. \cite[(4.1)]{BH13}):
			\begin{equation}\label{eqn:long-exact-sequence}
				\cdots\ \longrightarrow\ H^2\left(BL_X(G),\  \bb{C}^*\right)\ \longrightarrow
			\end{equation}
			$$
			\ker\left[H^2\left(\mgpar, \
			(\mc{O}^{^{an}}_{\mgpar})^*\right)\ \longrightarrow\
			H^0\left(BL_X(G),\ H^2\left(\mc{Q}_G^{par},\ (\mc{O}^{^{an}}_{\mc{Q}_G^{par}})^*\right)\right)\right]
			$$
			$$
			\ \longrightarrow\ H^1\left(BL_X(G),\ H^1\left(\mc{Q}^{par}_G,\ (\mc{O}^{^{an}}_{\mc{Q}_G^{par}})^*\right)\right).
			$$
			In \eqref{eqn:long-exact-sequence}, we have $H^2\left(BL_X(G),\ \bb{C}^*\right) \ =\ 0$  \cite[Proposition 3.4]{BH13}. Let us compute the last term in \eqref{eqn:long-exact-sequence}. 
			Each $\Pic(G/P_i)$ is a free abelian group of finite rank, and the same is true for the group
			$\Pic(\mc{Q}_G)$ \cite[Lemma 1.4]{BLS98}. Using these, \eqref{eqn:picard-group-of-limit}
			implies that $$H^1\left(\mc{Q}^{par}_G,\ (\mc{O}^{^{an}}_{\mc{Q}_G^{par}})^*\right)\ =\ \bb{Z}^t$$ 
			for some positive integer $t$. Since $H^1(BL_X(G), \ \bb{Z})  =\ 0$ \cite[Proposition 3.4]{BH13}, it follows that
			\begin{equation*}
				H^1\left(BL_X(G),\ H^1\left(\mc{Q}^{par}_G,\ (\mc{O}^{^{an}}_{\mc{Q}_G^{par}})^*\right)\right)\ =\ H^1\left(BL_X(G),\ \bb{Z}^t\right)\ = \ 0.
			\end{equation*}
			Moreover, as the space $BL_X(G)$ is connected,
			$$H^0\left(BL_X(G),\ H^2\left(\mc{Q}_G^{par},\ (\mc{O}^{^{an}}_{\mc{Q}_G^{par}})^*\right)
			\right)\ \simeq\  H^2\left(\mc{Q}_G^{par},\ (\mc{O}^{^{an}}_{\mc{Q}_G^{par}})^*\right).$$
			Therefore, the exact sequence in \eqref{eqn:long-exact-sequence} implies the following inclusion of abelian groups:
			$$
			H^2\left(\mgpar,\ (\mc{O}^{^{an}}_{\mgpar})^*\right)\ \hookrightarrow\  H^2\left(\mc{Q}_G^{par}, \
			(\mc{O}^{^{an}}_{\mc{Q}_G^{par}})^*\right),
			$$
			which implies that
			$$
			H^2\left(\mgpar,\ (\mc{O}^{^{an}}_{\mgpar})^*\right)_{torsion}\ 
			\hookrightarrow\  H^2\left(\mc{Q}_G^{par}, \ (\mc{O}^{^{an}}_{\mc{Q}_G^{par}})^*\right)_{torsion}.
			$$
			Now Lemma \ref{lem:brauer-group-of-qgpar} completes the proof.
		\end{proof}
		
		Finally, we prove that the algebraic Brauer group of $\mgpar$ vanishes, by comparing it with the analytic Brauer 
		group of $\mgpar$.
		
		\begin{theorem}\label{thm:brauer-group-vanishing}
			Let $G$ be a simple and simply connected complex affine algebraic group.
			Then $$\Br\left(\mgpar\right)\ \,=\ \,0.$$
		\end{theorem}
		
		\begin{proof}
			We have $H^2\left(\mgpar, \ (\mc{O}^{^{an}}_{\mgpar})^*\right)_{torsion}\ =\ 0$ by
			Proposition \ref{prop:analytic-brauer-group-vanishing}. We would now like to compare this group with
			the corresponding algebraic Brauer group. This will be done by showing that the two conditions in
			Proposition \ref{prop:two-conditions} are satisfied for the stack $\mgpar$.
			
			\begin{enumerate}[(1)]
				\item Using the facts that $H^1_B\left(F,\ \bb{Z}\right)\ =\ 0$ as well as that $F$ and $\mc{M}_G$ are
				connected (since $G$ is simply connected), Lemma \ref{lem:singular-cohomology} implies that
				\begin{align}
					H^2_B\left(\mgpar, \ \bb{Z}\right)\ & \simeq\ \left( H^2_B\left(\mc{M}_G,\ \bb{Z}\right)\otimes
					H^0_B\left(F,\ \bb{Z}\right)\right)\ \bigoplus\  \left(H^0_B\left(\mgpar, \ \bb{Z}\right)\otimes H^2_B\left(F,\ \bb{Z}\right)\right)\nonumber\\
					\ &\simeq\ H^2_B\left(\mc{M}_G,\ \bb{Z}\right)\ \bigoplus\ H^2_B\left(F, \ \bb{Z}\right)\label{eqn:eq-1}.
				\end{align}
				Moreover, by \cite[Theorem 1.1 and Proposition 8.7]{LS97} we have
				\begin{align}\label{eqn:eq-3}
					\Pic\left(\mgpar\right)\ \simeq\ \Pic\left(\mc{M}_G\right)\ \oplus\ \Pic(F).
				\end{align}
				Now, by \eqref{eqn:eq-1} and \eqref{eqn:eq-3}, the first Chern class map $\Pic(\mgpar)\,\longrightarrow\,
				H^2_B\left(\mgpar, \ \bb{Z}\right)$ splits as a direct sum of the Chern class maps $\Pic\left(\mc{M}_G\right)
				\,\longrightarrow\, H_B^2\left(\mc{M}_G,\ \bb{Z}\right)$ and $\Pic\left(F\right)\,\longrightarrow
				\,H^2_B\left(F,\ \bb{Z}\right)$. As both these maps are surjective
				(cf. \cite[Proposition 5.1 and Remark 5.2]{Te98}), we conclude  that the Chern class map $\Pic(\mgpar)\,
				\longrightarrow\, H^2_B\left(\mgpar, \ \bb{Z}\right)$ is also surjective; in other words, the first condition in Proposition \ref{prop:two-conditions} is satisfied for $\mgpar$.
				
				\item As $\Pic\left(\mgpar\right)\ \simeq\ \Pic\left(\mc{M}_G\right)\ \oplus\ \Pic(F)$ by \eqref{eqn:eq-3}, it is clear that the second condition of  Proposition \ref{prop:two-conditions} is satisfied for $\mgpar$ as well. 
			\end{enumerate}
			Thus, from the comparison isomorphisms in Proposition \ref{prop:two-conditions} together with Proposition \ref{prop:analytic-brauer-group-vanishing}, it follows that 
			\begin{align*}
				\Br\left(\mgpar\right)\ \simeq\ H^2\left(\mgpar, \ (\mc{O}^{^{an}}_{\mgpar})^*\right)_{torsion}\ =\ 0.
			\end{align*}
			This completes the proof.
		\end{proof}


\begin{thebibliography}{AAAA} 
			
			
			%
			\bibitem[BB]{BB23} I. Biswas and U. N. Bhosle, Brauer and Picard groups of moduli spaces of parabolic vector 
			bundles on a real curve, \textit{Comm. Algebra} \textbf{51} (2023), 3952--3964.
			
			\bibitem[BCD1]{BCD23} I. Biswas, S. Chakraborty and A. Dey, Brauer group of moduli stack of stable parabolic 
			$PGL(r)$-- bundles over a curve, \textit{Int. Jour. Math.} \textbf{35} No. 1 (2023).
			
			\bibitem[BCD2]{BCD24} I. Biswas, S. Chakraborty and A. Dey, Brauer group of moduli of parabolic symplectic 
			bundles, \textit{preprint}.
			
			\bibitem[BD]{BD11} I. Biswas and A. Dey, Brauer group of a moduli space of parabolic 
			vector bundles over a curve, \textit{J. K-theory} {\bf 8} (2011), 437--449.
			
			\bibitem[BHog]{BH10} I. Biswas and A. Hogadi, Brauer group of moduli spaces of
			$\text{PGL}(r)$-bundles over a curve, \textit{Adv. Math.} {\bf 225} (2010), 2317--2331.
			
			\bibitem[BHol]{BH13} I. Biswas and Y. I. Holla, Brauer group of moduli of principal bundles
			over a curve, \textit{Jour. reine angew. Math.} \textbf{677} (2013), 225--249.
			
			\bibitem[BM]{BM19} I. Biswas and F. Machu, On the direct images of parabolic vector 
			bundles and parabolic connections, \textit{J. Geom. Phys.} {\bf 135} (2019), 219--234.
			
			\bibitem[BLS]{BLS98} A. Beauville, Y. Laszlo and C. Sorger, The Picard group of the moduli of $G$-bundles on a 
			curve, \textit{Compos. Math.} \textbf{112} (1998), 183--216.
			
			\bibitem[BY]{BY99} H. Boden and K. Yokogawa, Rationality of the moduli space of Parabolic bundles, \textit{Jour. 
				Lond. Math. Soc.} {\bf 59} (1999), 461--478.
			
			\bibitem[Bu]{B04} A. S. Buch, Quantum cohomology of partial flag manifolds, \textit{Trans. Amer. Math. Soc.} 
			\textbf{357} (2004), 443--458.
			
			\bibitem[Ce]{C19} K. $\check{C}$esnavi$\check{c}$ius, Purity for the Brauer group, \textit{Duke Math. Jour.} {\bf 
				168} (2019), 1461--1486.
			
			
			\bibitem[Ha]{Ha} R. Hartshorne, {\it Algebraic geometry}, Graduate Texts in Mathematics, No. 52. Springer-Verlag,
			New York-Heidelberg, 1977.
			
			\bibitem[KN]{KN97} S. Kumar and M. S. Narasimhan, Picard group of the moduli spaces of $G$-bundles, 
			\textit{Math. Ann.} \textbf{308} (1997), 155--173.
			
			\bibitem[LS]{LS97} Y. Laszlo and C. Sorger, The line bundles on the moduli of parabolic $G$--bundles over curves 
			and their sections, \textit{Annales scientifiques de l'\'E.N.S.} 4e s\'erie, \textbf{30} (1997), 
			499--525.
			
			\bibitem[MS]{MS80} V. B. Mehta and C. S. Seshadri, Moduli of vector bundles on curves
			with parabolic structure, \textit{Math. Ann.} {\bf 248} (1980), 205--239.	
			
			\bibitem[Se]{Ses} C. S. Seshadri, {\it Fibr\'es vectoriels sur les courbes alg\'briques}, Notes 
			written by J.-M. Drezet from a course at the \'Ecole Normale Sup\'erieure, June 1980, 
			Ast\'erisque, 96. Soci\'et\'e Math\'ematique de France, Paris, 1982.
			{\it Ast\'erisque}, no. 96 (1982).
			
			\bibitem[Te]{Te98} C. Teleman, Borel--Weil--Bott theory on the moduli stack of $G$-bundles over a curve, 
			\textit{Invent. Math.} \textbf{134} (1998), 1--57.
			
			
		\end{thebibliography}
	\end{document}